\documentclass{amsproc}
\usepackage{amssymb,amsmath,amsthm}
\usepackage{hyperref}
\usepackage{color}
\usepackage{enumitem}

\newcommand{\Z}{\mathbb{Z}}

\newcommand{\cL}{\mathcal{L}}
\newcommand{\cR}{\mathcal{R}}

\setlist[itemize]{leftmargin=0.5cm}
\setlist[enumerate]{leftmargin=1.2cm}

 \newtheorem{thm}{Theorem}[section]
 
 \newtheorem{lem}[thm]{Lemma}
 \newtheorem{prop}[thm]{Proposition}
 \theoremstyle{definition}
 \newtheorem{defn}[thm]{Definition}
 \theoremstyle{remark}
 \newtheorem{rem}[thm]{Remark}
 \newtheorem{notn}[thm]{Notation}
 \newtheorem{example}[thm]{Example}
 \numberwithin{equation}{subsection}

\definecolor{zielony}{rgb}{0.5, 0.9, 0.1}
\definecolor{czerwony}{rgb}{0.9, 0.2, 0.1}
\definecolor{niebieski}{rgb}{0.3, 0.1, 0.9}

\begin{document}

\title[Phylogenetic invariants for $\Z_3$ scheme-theoretically]
 {Phylogenetic invariants for $\Z_3$ scheme-theoretically}

\author{ Maria Donten-Bury }

\address{Instytut Matematyki UW, Banacha 2, PL-02097 Warszawa}

\email{M.Donten@mimuw.edu.pl}



\date{July 7, 2015.}

\thanks{This research was partially supported by a grant of Polish National Science Center (2012/07/N/ST1/03202).}

%%% ----------------------------------------------------------------------

\begin{abstract}
We study phylogenetic invariants of general group-based models of evolution with group of symmetries $\Z_3$. We prove that complex projective schemes corresponding to the ideal $I$ of phylogenetic invariants of such a model and to its subideal $I'$ generated by elements of degree at most 3 are the same. This is motivated by a conjecture of Sturmfels and Sullivant~\cite[Conj.~29]{SS}, which would imply that $I = I'$.
\end{abstract}

%%% ----------------------------------------------------------------------
\maketitle
%%% ----------------------------------------------------------------------

\section{Introduction}

One of the most important questions in phylogenetic algebraic geometry, motivated by applications, is to determine the ideal of phylogenetic invariants, i.e. the ideal of polynomials vanishing on an algebraic variety corresponding to a model of evolution. It turns out that even determining the minimal degree in which this ideal is generated is a difficult problem. This question is often addressed in the setting of general group-based models of evolution (see e.g.~\cite[Sect.~8.10]{Ph}), i.e. models with an abelian group of symmetries. The simplest (but having very interesting properties, studied in~\cite{BW}) example of this class is the binary model, $G \simeq \Z_2$. The Kimura 3-parameter model with $G \simeq \Z_2\times \Z_2$ can be understood as its generalization, important from the point of view of motivation coming from computational biology, see e.g.~\cite{ES, SS}. Another small example is the model with $G \simeq \Z_3$, considered in this note. It is known that algebraic varieties associated with group-based models (i.e. their geometric models) are toric,~\cite{SSE, SS}. This class appears also in connection with theoretical physics, see~\cite{Man}.

In~\cite[Conj.~29]{SS} Sturmfels and Sullivant conjecture that the ideal of phylogenetic invariants for a group-based model with group of symmetries $G$ is generated in degree at most $|G|$. The authors give a proof for the binary model and provide some experimental data supporting the conjecture for small trees and groups. In~\cite{phylocomp} we analyze a few more examples with computational methods, and also suggest a geometric approach to the problem of determining phylogenetic invariants.

Let $I$ be the ideal of phylogenetic invariants for a tree $T$ and an abelian group $G$ and $I'$ be the ideal generated by the invariants in degree at most $|G|$. If~\cite[Conj.~29]{SS} is true, it implies that $I = I'$. Since comparing these two ideals is a difficult task, there have been a few attempts to compare geometric objects defined by them: projective schemes, sets of zeroes, or even sets of zeroes in the open orbit of a toric model. This last approach is presented in~\cite{CFSM}. The set-theoretical version of this conjecture for the class of equivariant models introduced in~\cite{DK} is considered in~\cite{Draisma2012}. In~\cite{constrdeg} Micha\l{}ek proves the scheme-theoretical version for the 3-Kimura model, and also that for a fixed abelian group $G$ there is a bound on the degree in which $I$ is generated, independent on the size of the tree.

The aim of this note is to give a combinatorial proof of the scheme-theoretic version for $G \simeq \Z_3$, using ideas similar to those presented in~\cite{constrdeg}. We work over the field of complex numbers~$\mathbb{C}$.

\begin{thm}\label{theorem}
For $G \simeq \Z_3$ and any tree $T$ the projective schemes defined by the ideal $I$ of phylogenetic invariants of the corresponding model and its subideal $I'$ generated by elements of degree 2 and 3 are the same.
That is, the saturation of $I'$ with respect to the irrelevant ideal is equal to $I$.
\end{thm}

Note that this result implies also the set-theoretic one: to check whether a point lies in the set of zeroes of the ideal of phylogenetic invariants for $\Z_3$ it is sufficient to see if the invariants of degree at most 3 vanish.

The structure of the paper is as follows. In the next section we describe the background for the problem and reduce the statement of Theorem~\ref{theorem} to a purely combinatorial one, which is then solved in sections~\ref{section_no_pairs}-\ref{section_two_indices}. The notion of phylogenetic invariants is recalled in sections~\ref{section_phylo} and~\ref{phylo_inv}. Section~\ref{section_saturation} is devoted to explaining the role of saturation in the main result, which leads to a reduction of Theorem~\ref{theorem} to a combinatorial problem, and in section~\ref{section_deletion} we describe a basic step in the proof. Several examples illustrating these ideas are provided.
Then, in section~\ref{subsection_outline}, we sketch the relations between sections~\ref{section_no_pairs}-\ref{section_two_indices}, containing the main part of the proof: we analyze there a few separate cases depending on the form of a chosen phylogenetic invariant.

\subsection*{Acknowledgements.} Thanks to Mateusz Micha\l{}ek for comments on preliminary versions of the proof, and to him and Weronika Buczy\'nska for discussions.

The author would also like to thank anonymous referees for many important suggestions for improving the presentation of results.

\section{Background and outline of the proof}

In this section we recall basic definitions concerning group-based models, ideals of phylogenetic invariants and their combinatorial presentation. We explain how using the saturation simplifies dealing with phylogenetic invariants for $\Z_3$ and give the outline for the remaining part of the proof of the main result.

\subsection{Phylogenetic models}\label{section_phylo}
Let us recall briefly that by a \emph{general group-based model} with a finite abelian group of symmetries $G$ we understand a triple $(T,W,\widehat{W})$, where $T$ is a tree, $W$ is the regular representation of $G$ and $\widehat{W}\simeq \mathrm{End}(W)^G$ is the space of $G$-invariant endomorphisms of $W$. With such a structure one can associate a geometric model: an algebraic variety, affine or projective. This can be done in two equivalent ways, see e.g.~\cite{SS, mateusz}. The first one is to consider the closure of the image of a parametrization map (see e.g.~\cite{BW, SS}), the second, which will be used here, is to take a toric variety corresponding to a certain lattice polytope (see~\cite{mateusz} and section~\ref{phylo_inv} below).

The \emph{ideal of phylogenetic invariants} of $(T,W,\widehat{W})$ is the ideal vanishing on the corresponding geometric model. The construction of phylogenetic invariants for group-based models (or, more generally, equivariant models) can be reduced to the case of \emph{claw trees} $K_{1,r}$, i.e. trees with one inner vertex and $r$ leaves, see e.g.~\cite[Sect.~5]{SS} and~\cite{sull}. Roughly speaking, the tree $T$ can be constructed by gluing claw trees along edges and~\cite[Thm~26]{SS} states that the ideal of phylogenetic invariants of $T$ is the sum of ideals generated by certain extensions of phylogenetic invariants of these claw trees to $T$ and an ideal generated by quadrics. Therefore, if the ideals of phylogenetic invariants for $G$ and all $K_{1,r}$ are generated in degree $d \geq 2$, then so is the ideal of phylogenetic invariants for $G$ and any tree $T$. Hence the general idea is that it is sufficient to prove degree bounds for claw trees. This applies also to the case where the saturation is used.

\begin{prop}\label{prop_reduction_to_claws}
The statement of Theorem~\ref{theorem} for the claw trees, i.e. $T = K_{1,r}$, where $r$ runs through $\mathbb{N}$, implies the general statement, for any tree~$T$.
\end{prop}

\begin{proof}
By~\cite[Thm~26]{SS} the ideal of phylogenetic invariants of a tree~$T$ is generated by quadrics and binomials which are extensions to~$T$ of invariants of all claw trees $K_{1,r}$ embedded (as a star of a vertex) in~$T$. The construction of these extensions is described in~\cite[Lem.~24]{SS} in terms of \emph{group-based flows}. The general idea of the proof is explained below without using this notion, but it is very useful for filling in the details. It suffices to show that if any binomial $B$ in the ideal of phylogenetic invariants of $K_{1,r}$ can be written as a sum of multiples $\{B_j \colon j \in J \}$ of invariants of degree at most $|G|$, then any its extension $\overline{B}$ to~$T$ also has such a decomposition. 

We can construct the decomposition of $\overline{B}$ inductively, choosing $B_{j_0}$ which has a common monomial with~$B$, extending it to some $\overline{B_{j_0}}$ on the whole tree~$T$, which has a monomial in common with~$\overline{B}$, and then repeating the operation for the binomial $\overline{B}-\overline{B_{j_0}}$ and the set $\{B_j \colon j \in J\setminus \{j_0\}\}$. And a suitable extension $\overline{B_{j_0}}$ comes in a natural (but not necessarily unique) way from the chosen extension $\overline{B}$ (this follows from the fact that $B$ and $B_{j_0}$ share the same monomial, hence also a multiset of group-based flows).

After performing this step for all $j \in J$ we obtain a phylogenetic invariant of~$T$ which must be a sum of multiples of quadric invariants. These quadrics correspond (in terms of group-based flows) to swapping data encoding extensions of invariants from $K_{1,r}$ to $T$, associated with trees $T_1,\ldots,T_r$ which are glued to $K_{1,r}$ along edges to obtain~$T$, similarly as in the proof of~\cite[Thm~26]{SS}. 
\end{proof}

Therefore, in the next sections we restrict to the case of claw trees.

Note that in the proof of Proposition~\ref{prop_reduction_to_claws} we do not use any information on the group~$G$, hence for any finite abelian group~$G$ it is sufficient to prove the scheme-theoretic version of~\cite[Conj.~29]{SS} just for claw trees.

\subsection{Combinatorial description of phylogenetic invariants}\label{phylo_inv}

We use the notation of~\cite[Sect.~3]{constrdeg} for phylogenetic invariants on a claw tree: we present them as relations between group-based flows on the tree (see~\cite[Def.~3.5]{constrdeg} and~\cite[Sect.~4]{SS}). We recall here the notion of group-based flows only in the case of claw trees.

\begin{defn}\label{flows}
A \emph{group-based flow} $F = (f_1,\ldots,f_r) \in G^r$ on $K_{1,r}$ is an assignment of elements $f_1.\ldots,f_r$ of $G$ to the edges $e_1,\ldots,e_r$ of $K_{1,r}$ such that $f_1+\ldots+f_r = 0$.
\end{defn}

By~\cite{mateusz}, group-based flows correspond to vertices of the lattice polytope $P$ describing the toric structure of the associated geometric model (as in~\cite[Chapter~2]{ToricBook}). Note that $P$ is naturally placed in a big lattice $\Lambda$ encoding all assignments of group elements to the edges of the tree: to each edge $e_i$ we assign $|G|$ coordinates, and the group element $f_i$ assigned to $e_i$ is described by setting all corresponding coordinates to 0 except of the one corresponding to $f_i$, which becomes 1. However, $P$ has to be considered in a sublattice $\Lambda_P$ spanned by its vertices to obtain the correct toric structure. 

\begin{notn}
From now on, let $I$ be the ideal of phylogenetic invariants on a chosen claw tree $K_{1,r}$ and $I'$ be the ideal generated by the invariants in degree at most~$|G|$. 
\end{notn}

The ideal $I$ of phylogenetic invariants is an ideal of $\mathbb{C}[x_{F_1},\ldots,x_{F_p}]$, where $F_1,\ldots,F_p$ are all vertices of $P$, i.e. all group-based flows on $K_{1,r}$. Its elements correspond to relations between vertices of $P$. More precisely, by~\cite[Lem.~4.1]{Sturmfelsksiazka} we can express generators of $I$ as relations $\sum L_i=\sum R_j$, where $\cL = \{L_i\}$ and $\cR = \{R_j\}$ are finite multisets of group-based flows of the same cardinality. We add flows coordinatewise in $G^r$, counting different group elements at each index (which corresponds to addition in the lattice $\Lambda$). In terms of polynomials, the relation $\sum L_i=\sum R_j$ corresponds to a binomial $\prod_{L_i \in \cL} x_{L_i} - \prod_{R_j \in \cR} x_{R_j}$. We assume $\cL\cap \cR = \emptyset$. 

\subsection{Saturation}\label{section_saturation}
We consider the saturation of $I'$ in the irrelevant ideal $J = (x_{F_1},\ldots,x_{F_p})$. The aim is to prove that $I':J^{\infty} = I$. Since~$I$ is prime and thus saturated, the nontrivial part is to show that~$I$ is contained in the saturation of~$I'$. Take $a \in I$. It belongs to $I':J^{\infty}$ if there is $n \in \mathbb{N}$ such that $aJ^n \subseteq I'$. This is equivalent to saying that for every variable $x_{F_q}$ there is $n_q \in \mathbb{N}$ such that $ax_{F_q}^{n_q} \subseteq I'$. Moreover, by~\cite[Lem.~4.3]{constrdeg} the situation is symmetric with respect to permutations of variables, hence it is sufficient to prove this statement for one chosen variable. We may choose the one corresponding to the trivial flow, which will be denoted by 0. In terms of relations between multisets of flows multiplying an element of $I$, $\prod_{L_i \in \cL} x_{L_i} - \prod_{R_j \in \cR} x_{R_j}$, by a power of this variable is adding the same number of copies of the trivial flow to $\cL$ and $\cR$. Let us explain how this allows to simplify the structure of phylogenetic invariants. We concentrate on the case of $G\simeq \Z_3$.

\begin{lem}\label{flow_types} Let $a$ be an element of the ideal $I$ of phylogenetic invariants for $\Z_3$, represented by multisets of flows $\cL$ and $\cR$. Then there is $n \in \mathbb{N}$ such that $x_0^na$ can be decomposed as a sum of an element of $I'$ and an element of $I$ represented by $\cL'$ and $\cR'$ which consist only of flows of two types:
\begin{enumerate}
\item pairs -- flows with only two nontrivial entries,
\item triples -- flows with only three nontrivial entries.
\end{enumerate}
\end{lem}

\begin{proof}
For any $\Z_3$-flow $L_1 = (f_1,\ldots,f_r)$ there are $k,m \leq r$ such that in $(f_k,\ldots,f_m)$ at least two and at most three entries are nonzero and $f_k+\ldots+f_m = 0$. We use a quadric relation involving the trivial flow $L_1 + 0 = L_1' + L_1''$ to divide $L_1$ into $L_1' = (0,\ldots,0,f_k,\ldots,f_m,0,\ldots,0)$ and $L_1'' = (f_1,\ldots,f_{k-1},0,\ldots,0,f_{m+1},\ldots,f_r)$.
 In terms of polynomials this gives us 
$$x_0(\prod x_{L_i} - \prod x_{R_j}) = \big[ (x_{L_1}x_0 - x_{L_1'}x_{L_1''})\prod_{i\neq 1} x_{L_i}\big] + \big[x_{L_1'}x_{L_1''}\prod_{i\neq 1} x_{L_i} - x_0\prod x_{R_i}\big].$$
That is, we have written the chosen element $a \in I$ multiplied by $x_0$ as a sum of two binomials. The first one is a multiple of a quadric, hence lies in $I'$, so to prove that the chosen element is in $I'$ it is sufficient to show that the second one is also in $I'$ (possibly after multiplying by a power of $x_0$).

Thus in terms of flows, instead of working with $\sum L_i + 0= \sum R_j + 0$ we can now consider the relation $$L_1' + L_1'' + \sum_{i\neq 1} L_i = \sum R_j + 0,$$
that is we have replaced $L_1 + 0$ with $L_1'+L_1''$, where $L_1'$ is already a pair or a triple. To finish, we use the induction on the size of $L_1$, i.e. the number of nontrivial group elements in this flow, and then apply the same argument for other flows in~$\cL$ and~$\cR$.
\end{proof}

\subsection{Basic step: using a relation and deleting a flow}\label{section_deletion}

\begin{notn}\label{replacing}
The operation of writing a binomial as a sum of an element of $I'$ and another binomial with a simpler structure is an elementary tool in the proof. We will refer to this operation as \emph{replacing a multiset of flows in $\cL$ or $\cR$ with a different multiset of flows}, or just as \emph{using the relation}.
\end{notn}

The idea of decomposing flows into simpler ones after adding the trivial flow works analogously for any finite abelian group, see~\cite[Lem.~7.2]{mateusz}, but in the case of $\Z_3$ (and other groups of small order) the set of possible basic configurations is relatively simple and allows us to proceed with a combinatorial argument.

\begin{notn}\label{notation_pairs_triples}
Let $g_1$, $g_2$ be the nontrivial elements of $\Z_3$. There is one type of pairs and two types of triples:
\begin{enumerate}
\item $g_1$ and $g_2$ assigned to two chosen indices $a$ and $b$ respectively -- will be denoted $(a,b)$;
\item $g_1$ assigned to three chosen indices $a$, $b$ and $c$ -- will be denoted $(a,b,c)_{g_1}$;
\item $g_2$ assigned to three chosen indices $a$, $b$ and $c$ -- will be denoted $(a,b,c)_{g_2}$.
\end{enumerate}
\end{notn}
Note that for a pair the order of $a$ and $b$ is meaningful, while for triples a permutation of $(a,b,c)$ does not change anything. Obviously, indices in a pair or a triple must be different.

Moreover, we show that one may assume that on one side there are only triples of one type.

\begin{lem}\label{different_triple_types}
If $\cR$ or $\cL$ contains both a $g_1$-triple $(a,b,c)_{g_1}$ and a $g_2$-triple $(x,y,z)_{g_2}$, we may replace them with pairs (possibly after adding a trivial flow to both sides of the relation).
\end{lem}

\begin{proof}
Without loss of generality assume that $L_1 = (a,b,c)_{g_1}$ and $L_2 = (x,y,z)_{g_2}$.
We use the cubic relation
$$(a,b,c)_{g_1} + (x,y,z)_{g_2} + 0 = (a,x) + (b,y) + (c,z).$$
This may require adding a trivial flow to both sides of the relation if there is no trivial flow in~$\cL$. Also, it may be necessary to permute entries of a triple so that we obtain sensible pairs (i.e. with different indices), but such a permutation always exists. 
Thus we may replace the left hand side of this cubic relation by its right hand side in $\cL$, as explained in Lemma~\ref{flow_types} and Notation~\ref{replacing}.
After a finite number of steps we obtain the relation which has only $g_1$-triples or only $g_2$-triples.
\end{proof}

\begin{example}\label{example_1}
We illustrate the reduction to triples and pairs, as in Lemma~\ref{flow_types}, with an example. Consider a phylogenetic invariant on $K_{1,5}$ given by (multi)sets of flows $\cL = \{L_1,\ldots,L_4\}$ and $\cR = \{R_1,\ldots,R_4\}$, where
\begin{center}
\begin{tabular}{ll}
$L_1 = (g_1,g_2,g_2,g_1,0)$, & $R_1 = (g_1,0,0,g_1,g_1) = (1,4,5)_{g_1}$, \\ 
$L_2 = (0,g_1,g_2,g_1,g_2)$, & $R_2 = (g_2,0,g_2,0,g_2) = (1,3,5)_{g_2}$, \\ 
$L_3 = (g_2,0,0,0,g_1) = (5,1)$, & $R_3 = (0,g_2,0,g_1,0) = (4,2)$,\\
$L_4 = (0,0,0,0,0) = 0$, & $R_4 = (0,g_1,g_2,0,0) = (2,3)$. \\
\end{tabular}
\end{center}
Where possible, we write flows (elements of $\Z_3^5$) already using Notation~\ref{notation_pairs_triples} for pairs and triples; leaves of $K_{1,5}$ are indexed by $\{1,\ldots,5\}$. One can check that this is indeed a phylogenetic invariant, since to each leaf of $K_{1,5}$ the same multiset of elements of $\Z_3$ is assigned by~$\cL$ and~$\cR$.

Let $x_0$ be the variable corresponding to the trivial flow. We have to prove that the element of~$I$ represented by $\cL$ and $\cR$ can be multiplied by $x_0^n$ for some $n \in \mathbb{N}$ such that the result belongs to $I'$, generated by elements of degree at most~3. In terms of flows, we show that after adding $n$ copies of the trivial flow to both $\cL$ and $\cR$ we may decompose the relation between flows in these multisets as a sum of relations having at most~3 flows on each side.

We need to use two relations in~$\cL$ to decompose $L_1$ and $L_2$:
\begin{align*}
L_1 + 0 &= (g_1,g_2,g_2,g_1,0) + 0 = (g_1,g_2,0,0,0) + (0,0,g_2,g_1,0) = (1,2) + (4,3),\\
L_2 + 0 &= (0,g_1,g_2,g_1,g_2) + 0 = (0,g_1,g_2,0,0) + (0,0,0,g_1,g_2) = (2,3) + (4,5).
\end{align*}

Since both require the trivial flow and $\cL$ contains just one, we add one trivial flow to both sides (thus $n\geq 1$). After these operations both sides contains $(2,3) = R_4$. We just eliminate it and in the next step we will consider disjoint multisets:
\begin{align*}
\cL' &\hbox{ containing } L_3 = (5,1), L_5 = (1,2), L_6 = (4,3), L_7 = (4,5),\\
\cR' &\hbox{ containing } R_1 = (1,4,5)_{g_1}, R_2 = (1,3,5)_{g_2}, R_3 = (4,2), R_5 = 0.
\end{align*}

In terms of polynomials, we start from $x_{L_1}x_{L_2}x_{L_3}x_0 - x_{R_1}x_{R_2}x_{R_3}x_{R_4}\in I$ and multiply it by $x_0$. The relations reducing $L_1$ and $L_2$ are quadric binomials
$$x_{L_1}x_0 - x_{L_5}x_{L_6} \hbox{ and } x_{L_2}x_0 - x_{L_7}x_{R_4}.$$
Thus we just need to prove that
\begin{multline*}
x_{L_1}x_{L_2}x_{L_3}x_0^2 - x_{R_1}x_{R_2}x_{R_3}x_{R_4}x_0 - x_{L_2}x_{L_3}x_0(x_{L_1}x_0 - x_{L_5}x_{L_6})-\\ - x_{L_3}x_{L_5}x_{L_6}(x_{L_2}x_0 - x_{L_7}x_{R_4}) = 
x_{R_4}(x_{L_3}x_{L_5}x_{L_6}x_{L_7} - x_{R_1}x_{R_2}x_{R_3}x_0)
\end{multline*}
is in~$I'$, and in fact this is true for $x_{L_3}x_{L_5}x_{L_6}x_{L_7} - x_{R_1}x_{R_2}x_{R_3}x_0$, i.e. we can forget~$x_{R_4}$.
\end{example}

\begin{example}\label{example_2}
We continue the previous example to  present the elimination of $g_1$-triples and $g_2$-triples appearing on one side, as in Lemma~\ref{different_triple_types}. In the relation between $\cL'$ and $\cR'$ there is a $g_1$-triple and a $g_2$-triple in $\cR'$. We use the relation
$$R_1 + R_2 + R_5 = (1,4,5)_{g_1} + (1,3,5)_{g_2} + 0 = (1,3) + (4,5) + (5,1)$$
(which does not require adding the trivial flow to the relation). Note that now pairs $L_7 = (4,5)$ and $L_3 = (5,1)$ are on both sides of the relation, hence we may eliminate them and consider only $\cL''$ containing $L_5 = (1,2), L_6 = (4,3)$ and  $\cR''$ containing  $R_3 = (4,2), R_6 = (1,3)$. But $\cL''$ and $\cR''$ already represent a quadric relation 
$$(4,2) + (1,3) = (1,2) + (4,3).$$
Thus we obtain that the relation from Example~\ref{example_1} multiplied by $x_0$ can be decomposed into a sum of multiples of quadrics and cubics, hence actually represents an element of~$I'$.

In terms of polynomials, the cubic binomial corresponding to changing triples into pairs in $\cR$ is
$$x_{R_1}x_{R_2}x_0 - x_{R_6}x_{L_3}x_{L_7}.$$
Hence we would like to prove that the following polynomial is in $I'$:
\begin{multline*}
x_{L_3}x_{L_5}x_{L_6}x_{L_7} - x_{R_1}x_{R_2}x_{R_3}x_0 + x_{R_3}(x_{R_1}x_{R_2}x_0 - x_{R_6}x_{L_3}x_{L_7}) =\\
=  x_{L_3}x_{L_5}x_{L_6}x_{L_7} - x_{R_3}x_{R_6}x_{L_3}x_{L_7} = x_{L_3}x_{L_7}(x_{L_5}x_{L_6} - x_{R_3}x_{R_6}),
\end{multline*}
but it is clearly a multiple of a quadric binomial.
\end{example}

We are going to prove that any binomial in~$I$, represented by a relation between multisets of flows~$\cL$ and~$\cR$, can be written as a sum of multiples of quadric and cubic binomials, which belong to~$I'$. This will be done by simplifying inductively the relation between $\cL$ and $\cR$ using relations between multisets of flows of cardinality at most~3.

As one can already see in the examples above, the aim of using a relation on multisets of flows~$\cL$ and~$\cR$, as in Notation~\ref{replacing}, is to obtain $\cL'$ and $\cR'$ with non-empty intersection. Then the flow which appears on both sides of the relation can be deleted, which allows us to use induction. Let us explain this operation more carefully.

\begin{notn}\label{notation_flow_reduction}
By \emph{deleting a flow} $F = L_1 = R_1$, which belongs both to $\cL$ and $\cR$, we mean that instead of the relation $F + \sum_{i \neq 1} L_i = F + \sum_{j\neq 1} R_j$ we consider the relation $\sum_{i \neq 1} L_i = \sum_{j\neq 1} R_j$. In terms of polynomials this is equivalent to saying that to prove that the binomial $x_F(\prod_{i\neq 1}x_{L_i} - \prod_{j\neq 1}x_{R_j})$ is in $I'$ it suffices to show that $\prod_{i\neq 1}x_{L_i} - \prod_{j\neq 1}x_{R_j} \in I'$.
By such a deletion we obtain a relation of smaller degree with respect to the grading with will be used in the induction.
\end{notn}

We show the decomposition of an element of~$I$ into multiples of quadrics and cubics using relations and deletions of flows in a simple example.

\begin{example}\label{example_2}
We consider a relation already in the form of multisets of triples and pairs, without different types of triples on one side. Several steps of reduction are shown, this time only in the notation of group-based flows. Take again $K_{1,5}$ and a relation
\begin{center}
\begin{tabular}{llllll}
$\cL\colon\quad$ & $L_1 = (1,2,3)_{g_1}$, & $L_2 = (4,2)$, & $L_3 = (1,5)$, & $L_4 = (5,3)$, & $L_5 = (4,3)$, \\ 
$\cR\colon\quad$ & $R_1 = (3,4,5)_{g_1}$, & $R_2 = (1,2)$, & $R_3 = (1,3)$, & $R_4 = (4,5)$, & $R_5 = (2,3)$. \\
\end{tabular}
\end{center}
It belongs to the most general case described in section~\ref{section_four_indices}. We start from using a quadric relation in $\cR$:
$$R_4 + R_5 = (4,5) + (2,3) = (4,3) + (2,5) = L_5 + (2,5).$$
In this way we obtain $L_5 = (4,3)$ on the right side of the relation, which allows us to reduce the problem to considering $\cL' = \{L_1,L_2,L_3,L_4\}$ and $\cR' = \{R_1,R_2,R_3,(2,5)\}$. Now we can again use a relation in $\cR$:
$$R_3 + (2,5) = (1,3) + (2,5) = (1,5) + (2,3) = L_3 + (2,3),$$
hence $L_3$ can be deleted and we are left with $\cL''=\{L_1,L_2,L_4\}$ and $\cR'' = \{R_1,R_2,(2,3)\}$. These sets represent a cubic relation, i.e. we have decomposed the chosen element of~$I$ as a sum of multiples of two quadrics and a cubic. However, we may do one more step to see that this is in fact a sum of multiples of four quadrics. We use a relation involving a triple in $\cL$:
$$L_1 + L_2 = (1,2,3)_{g_1} + (4,2) = (2,3,4)_{g_1} + (1,2) = (2,3,4)_{g_1} + R_2.$$
Now we can delete $R_2=(1,2)$ and obtain a quadric relation
$$(2,3,4)_{g_1} + (5,3) = (3,4,5)_{g_1} + (2,3).$$
\end{example}

\subsection{Outline of the argument}\label{subsection_outline}
The general idea of the argument is induction, but not on the standard degree, because we need to multiply the relation by the variable corresponding to the trivial flow in Lemmata~\ref{flow_types} and~\ref{different_triple_types}. We use the grading by the size of the flow: a variable is in the $d$-th graded piece if the corresponding flow has $d$ nontrivial entries. We will take a relation and decompose it into a sum of relations in $I'$, i.e. relations of (standard) degree 2 or 3 multiplied by monomials, and relations of smaller degree with respect to the grading introduced above. Before starting the induction we change the relation such that there are no $g_1$-triples and $g_2$-triples together on one side, by Lemma~\ref{different_triple_types}. We will always assume that this is satisfied and ensure that our modifications do not violate this condition. 

In the following sections we consider several separate cases depending on possible configurations of elements in~$\cL$ and $\cR$. We start from analyzing the situation where there are no pairs at all in the relation, see section~\ref{section_no_pairs}. In the next two sections we assume that $\cL$ or $\cR$ contains at least two different pairs, e.g. $(1,2)$ and $(a,b)$. There are a few possibilities. In section~\ref{section_four_indices} these pairs consist of four different indices. Section~\ref{section_three_indices} concerns the case where these pairs has only three different indices, i.e. the set $\{1,2,a,b\}$ has three elements. It is divided into two subsections based on the configuration of indices: in~\ref{subsection_three_indices_1} we look at the case $1=a$ and in~\ref{subsection_three_indices_2} at the case $1=b$ (these are all possibilities up to a permutation). In the second one we either reduce the relation or prove that in fact we are in the situation of~\ref{subsection_three_indices_1}. The last possibility, described in section~\ref{section_two_indices}, is that only two different indices can appear in pairs on each side of the relation.

\section{No pairs}\label{section_no_pairs}
First assume that $\cL$ and $\cR$ consist only of triples. By Lemma~\ref{different_triple_types}, without loss of generality we may assume that there are only $g_1$-triples (we will use the symmetry between $g_1$ and $g_2$, and also between $\cL$ and $\cR$, all the time).

It is worth noting that the fact that in this case the relations can be generated in degree 2 is a consequence of a result on uniform (or, much more generally, strongly base orderable) matroids. This is a special case of the White's conjecture, for the details see~\cite{matroids} and references therein.

We consider two subcases, described in sections~\ref{nopairs_2ind} and~\ref{nopairs_no2ind}, depending on whether we can find triples on different sides of the relation which differ just by one index.

\subsection{There are $g_1$-triples, one in~$\cL$ and one in~$\cR$, with two indices in common}\label{nopairs_2ind} Say $(1,2,3)_{g_1} \in \cL$ and $(1,2,a)_{g_1}\in \cR$ -- we rename the indices if necessary, as we will do often throughout the proof. Then there must be a triple $(3,b,c)_{g_1}\in \cR$ for some indices $b,c$, because $g_1$ appears in~$\cL$ at the index~$3$ at least once. If $a\neq b$ and $a\neq c$, then we can use the relation $$(1,2,3)_{g_1} + (a,b,c)_{g_1}= (1,2,a)_{g_1} + (3,b,c)_{g_1}$$ in~$\cR$ and delete the flow $(1,2,3)_{g_1}$, as explained in Notation~\ref{replacing} and Notation~\ref{notation_flow_reduction}.

Thus we only have to consider the situation where $a \in \{b,c\}$. Without loss of generality $a=c$, so $(3,a,b)_{g_1}\in \cR$ for some $b$. Moreover, we may assume that 3 can appear in $\cR$ only in $g_1$-triples containing also $a$ -- otherwise we could find a triple which might be used in the relation above.

Symmetrically, if $g_1$ appears on the index $a$ in $\cL$ in a triple $(a,d,e)_{g_1}$ then either we can use the relation 
\begin{equation}\label{rel_counting_arg}
(1,2,3)_{g_1} + (a,d,e)_{g_1}= (1,2,a)_{g_1} + (3,d,e)_{g_1}
\end{equation}
in $\cL$ and delete $(1,2,a)_{g_1}$ or we have $3 \in \{d,e\}$. We prove by contradiction that it is impossible that for every triple $(a,d,e)_{g_1}$in $\cL$ we have $3 \in \{d,e\}$. Hence assume that~$a$ appears in $\cL$ only in $g_1$-triples containing~3.

We finish with an argument which will be repeated frequently throughout the proof.
Assume that there are precisely $n$ $g_1$-triples containing $3$ in $\cR$. We have shown above that in considered situation all of them must contain~$a$. Then $a$ appears at least $n$ times with $g_1$ in $\cR$ (in all~$n$ triples with~3, but maybe also in other flows). Since~$\cL$ and~$\cR$ are two sides of a relation, $a$ appears also at least $n$ times with~$g_1$ in $\cL$. But this implies that~3 appears at least $n+1$ times with $g_1$ in $\cL$: once in the triple $(1,2,3)_{g_1}$ not containing $a$, and $n$ times in a triple containing~$a$. This is a contradiction, since the numbers of appearances of $g_1$ on a chosen index on both sides must be equal.

Thus we conclude that it is impossible that $a$ appears only in $g_1$-triples containing~3 in $\cL$. That is, there is a triple $(a,d,e)_{g_1}$ in $\cL$ such that $3 \notin \{d,e\}$, so we can use the relation~\eqref{rel_counting_arg} to obtain a deletion of flows, as explained in Notation~\ref{notation_flow_reduction}.

Let us reformulate this argument in a form of a general observation, so we can refer to it later~on.

\begin{lem}\label{counting_argument}
Let $g_k$ and $g_m$ be non-trivial elements of~$\Z_3$ (not necessarily different), and $\alpha, \beta \in \{1,\ldots,r\}$ be indices of leaves of~$K_{1,r}$, where $\alpha\neq \beta$. Assume that $\cL$ and $\cR$ consist only of pairs and triples, and 
\begin{itemize}
\item if $\alpha$ appears in $\cL$ in a flow with $g_m$, then this flow contains also $\beta$ with $g_k$,
\item if $\beta$ appears in $\cR$ in a flow with $g_k$, then this flow contains also $\alpha$ with $g_m$.
\end{itemize}
Then $\beta$ cannot appear in $\cL$ with $g_k$ in a flow which does not contain $\alpha$ with $g_m$. 

Obviously, $\cL$ and $\cR$ may be swapped if needed.
\end{lem}

\begin{proof}
The idea is to count appearances of $\alpha$ with $g_m$ and $\beta$ with $g_k$, exactly as in the example above -- let us explain the details. 

Assume that there is a flow in $\cL$ which contains $\beta$ with $g_k$, but not $\alpha$ with $g_m$. By $n$ we denote the number of configurations containing $\beta$ with $g_k$ in $\cR$. By assumptions of the lemma every such a configuration contains also $\alpha$ with $g_m$. Hence $\alpha$ appears in~$\cR$ with~$g_m$ at least $n$ times. Since the sums of multisets of flows $\cR$ and $\cL$ and $\cR$ are equal, there are also at least $n$ occurrences of $\alpha$ with $g_m$ in $\cL$. But, by assumptions of the lemma, this implies that $\beta$ appears at least $n+1$ times with $g_k$ in $\cL$: $n$ times with $\alpha$ with $g_m$ and also in a flow which does not contain $\alpha$ with $g_m$. This contradicts the fact that $\cL$ and $\cR$ are multisets of flows with equal sums. Hence, under the assumptions of the lemma, $\cL$ cannot contain any configuration containing $\beta$ with $g_k$ but not containing $\alpha$ with $g_m$.

For example, in the argument given before the lemma we have $\alpha = a$, $\beta = 3$ and $k=m=1$. We consider the situation where~$a$ appears with~$g_1$ in~$\cL$ only in triples containing~3, and~3 appears with $g_1$ in $\cR$ only in triples containing~$a$. We get a contradiction with the fact that the triple $(1,2,3)_{g_1}$, containing~3 but not containing~$a$, is in~$\cL$, thus showing that such a situation is impossible.
\end{proof}

\subsection{There are no two triples, one in~$\cL$ and one in~$\cR$, with two indices in common}\label{nopairs_no2ind} 
Take $(1,2,3)_{g_1} \in \cL$ (if necessary, we rename the indices). Then there is some $(1,a,b)_{g_1}\in \cR$, because~1 must appear also in~$\cR$ with $g_1$. If there is $(2,c,d)_{g_1}\in \cR$ such that $\{a,b\}\neq\{c,d\}$, then we can
find a relation which gives a reduction to the previous case~\ref{nopairs_2ind}. For example, if $a\notin \{c,d\}$, then the relation used is $(1,a,b)_{g_1} + (2,c,d)_{g_1} = (1,2,b)_{g_1} + (a,c,d)_{g_1}$, and we obtain the triple $(1,2,b)_{g_1}\in \cR$, which has two indices in common with $(1,2,3)_{g_1} \in \cL$. Hence we may assume that 2 with $g_1$ appears always in the triple $(2,a,b)_{g_1}$ in $\cR$. Using the same argument we prove that 3 with $g_1$ appears always in the triple $(3,a,b)_{g_1}$ in~$\cR$.

Note that since $(1,a,b)_{g_1}$ is in~$\cR$, then $a$ and $b$ must appear in $g_1$-triples in $\cL$. If there is a triple $(a,e,f)_{g_1} \in \cL$ and $\{e,f\} \neq \{2,3\}$, then again we can find a quadric relation which leads to the previous case~\ref{nopairs_2ind}. Hence we may assume that $a$ appears with $g_1$ in $\cL$ always in $(a,2,3)_{g_1}$ and $b$ in $(b,2,3)_{g_1}$. We finish in a very similar way as in Lemma~\ref{counting_argument}. Let $n$ be the number of $g_1$-triples containing index 2 or 3 in~$\cR$. Since in considered situation every such a triple contains also $a$ and $b$, $\cR$ contains $n$ $g_1$-triples with $a$ and $b$. But we have assumed that $a$ and $b$ appear in $\cL$ only in triples containing also 2 and 3. Hence 2 (and 3 also) appears in $\cL$ with $g_1$ at least $2n$ times: $n$ times with $a$ and $n$ times with $b$. This is a contradiction with the fact that~2 appears in~$\cR$ at most $n$ times. 

Summing up, the only possibility in this case is that a relation can be found which gives a reduction to the previous one,~\ref{nopairs_2ind}.

\section{At least two pairs on one side, four different indices}\label{section_four_indices}

Now assume that there exist pairs $(1,2)$ and $(a,b)$ in $\cL$ such that $1, 2, a, b$ are pairwise different indices (we swap $\cL$ with $\cR$ if necessary).

\begin{rem}\label{exchanged_pairs}
If two pairs $(\alpha,\beta)$ and $(\gamma,\delta)$ with all indices different occur on one side of the relation, then we may swap their elements, that is replace them with $(\alpha,\delta)$ and $(\gamma,\beta)$ by applying quadric relation $(\alpha,\beta)+(\gamma,\delta) = (\alpha,\delta)+(\gamma,\beta)$.
\end{rem}

There are three cases to consider (up to swapping $g_1$ and $g_2$):
\begin{enumerate}
\item[(4.1)] $\cR$ contains triples $(1,x,y)_{g_1}$ and $(a,s,t)_{g_1}$ for some indices $x,y,s,t$ such that $x\neq y$ and $s\neq t$ (thus, by Lemma~\ref{different_triple_types}, neither $2$ nor $b$ is contained in a $g_2$-triple in $\cR$);
\item[(4.2)] 1 is contained in a $g_1$-triple in $\cR$ and $a$ does not appear in any $g_1$-triple in $\cR$ (thus, by Lemma~\ref{different_triple_types}, neither 2 nor $b$ appears in a $g_2$ triple in $\cR$);
\item[(4.3)] there are no $g_1$-triples containing 1 or $a$ and no $g_2$-triples containing 2 or $b$ in $\cR$.
\end{enumerate}

\subsection{There are $(1,x,y)_{g_1}$, $(a,s,t)_{g_1}$, $(u,2)$ and $(v,b)$ in $\cR$} Such two pairs must be in~$\cR$ because 2 and $b$ cannot occur in $g_2$-triples. If $u \notin \{x,y\}$ or $v\notin \{s,t\}$ then we could use a quadric relation, e.g. $(1,x,y)_{g_1}+(u,2) = (u,x,y)_{g_1} + (1,2)$, to get $(1,2)$ or $(a,b)$ in $\cR$ and delete a pair. Hence, without loss of generality, we consider only the situation where $u=x$ and $v=s$, that is, $\cR$ contains $(x,2)$ and $(s,b)$. We see that $x$ must appear with~$g_1$ in~$\cL$.

\subsubsection{} First assume that $x$ occurs with $g_1$ in $\cL$ only in triples and take such a triple $(x,c,d)_{g_1} \in \cL$. Then either we can use the relation $(1,2) + (x,c,d)_{g_1} = (x,2)+ (1,c,d)_{g_1}$ and delete $(x,2)$ or we may assume that $c=1$ and $x$ appears with $g_1$ in $\cL$ always in a triple containing~1. In the latter case~1 appears at least twice in $\cL$ with $g_1$. Hence the triple $(1,x,y)_{g_1}$ cannot be the only element of $\cR$ containing~1 with~$g_1$.

Assume that 1 appears with $g_1$ in a pair $(1,e) \in \cR$. If $x \neq e$ then we apply the quadric relation $(x,2) + (1,e) = (x,e) + (1,2)$ in $\cR$ and delete $(1,2)$. And if $s \neq e$, we can use the relation $(s,b) + (1,e) = (1,b) + (s,e)$ in $\cR$ and delete $(1,b)$ after swapping elements of pairs in $\cL$ as in Remark~\ref{exchanged_pairs}. Thus we can work under assumption that $x=s=e$, which means that $(a,x,t)_{g_1} \in \cR$. But then we may use the relation $(1,x) + (a,x,t)_{g_1} = (t,x)+ (a,x,1)_{g_1}$, because $t \neq x$ and $1\neq x$, since they appear together in some flows. Thus we may assume that $(a,x,1)_{g_1} \in \cR$. And in $\cL$ either we use the relation $(a,b) + (x,1,d)_{g_1} = (x,b) + (a,1,d)_{g_1}$, which allows us to delete $(x,b)$, or $a=d$ and $(x,1,a)_{g_1}$ can be deleted.

If 1 appears with $g_1$ in $\cR$ only in triples, we have $(1,e,f)_{g_1} \in \cR$. Then either we can use the relation $(x,2) + (1,e,f)_{g_1} = (1,2)+ (x,e,f)_{g_1}$ and delete $(1,2)$ or we may assume that $e=x$. That is, $1$ always appears in $\cR$ in triples containing also~$x$. To finish, we apply Lemma~\ref{counting_argument}: $x$ appears always in a $g_1$-triple containing 1 in $\cL$ and 1 in a $g_1$-triple containing $x$ in $\cR$, but also $(1,2) \in \cL$. Hence if we compare numbers of occurrences of~$1$ on both sides, we get a contradiction.

Thus in this case we are always able to use relations leading to a situation in which a pair or a triple can be deleted, so that the degree of the relation in considered grading is decreased.

\subsubsection{} The second possibility is that $x$ appears with $g_1$ in $\cL$ in a pair $(x,c)$. If $c \neq 1$ then we use the relation $(x,c) + (1,2) = (x,2) + (1,c)$ and delete $(x,2)$, so we may assume that $(x,1) \in \cL$. Then 1 appears in $\cR$ with $g_2$, and it can be only in pairs by Lemma~\ref{different_triple_types}, since there are already some $g_1$-triples. Then $(e,1) \in \cR$, and in fact we may assume that this pair is equal to $(2,1)$, because otherwise we would use the relation $(e,1) + (x,2) = (x,1)+(e,2)$ and delete $(x,1)$.

Now we try to apply the cubic relation $$(x,2)+(s,b)+(2,1) = (x,1)+(2,b)+(s,2)$$ in $\cR$ and delete $(x,1)$. This fails only if $s=2$. Then, in particular, $(2,b) \in \cR$. In this case 2 must appear with~$g_1$ in~$\cL$. If a pair $(2,h) \in \cL$ then we try to swap elements with either $(x,1)$ or $(a,b)$ (as described in Remark~\ref{exchanged_pairs}) and delete $(2,1)$ or $(2,b)$. This fails only if $x=h=a$, but then $(a,2) \in \cR$ and the same pair can be obtained in $\cL$ by swapping elements of $(1,2)$ and $(a,b)$.
And if there is a triple $(2,i,j)_{g_1} \in \cL$, then we try to use one of the relations
$$(2,i,j)_{g_1} + (a,b) = (a,i,j)_{g_1}+(2,b) \qquad (2,i,j)_{g_1} + (x,1) = (x,i,j)_{g_1}+(2,1)$$ and delete $(2,b)$ or $(2,1)$.

This is impossible only if $\{i,j\}\supseteq\{x,a\}$. If $x=a$ then we swap elements of $(1,2)$ and $(a,b)$ in $\cL$ and delete $(a,2)$. Hence we may assume that $x\neq a$ and there is $(2,x,a)_{g_1} \in \cL$. In $\cR$ we use the relation $(a,2,t)_{g_1}+(x,2) = (2,x,a)_{g_1}+(t,2)$ (note that $t \neq 2$ since they appear in the same triple), which allows us to delete $(2,x,a)_{g_1}$.

Thus again, in this case we can find relations which lead to a reduction of a flow.

\subsection{Now 1 is contained in some $g_1$-triple in $\cR$, say $(1,x,y)_{g_1}$, $a$ does not appear in any $g_1$-triple and 2 and $b$ do not appear in any $g_2$-triple in~$\cR$} Then $a$ appears with $g_1$ and 2 and $b$ with $g_2$ in $\cR$ only in pairs. Let $(z,2)$ be such a pair for the index 2 with $g_2$. Then we may assume that $z \in \{x,y\}$, because otherwise we could use the relation $(1,x,y)_{g_1} +(z,2) = (z,x,y)_{g_1} + (1,2)$ and delete $(1,2)$; let $z=x$. Take now $(a,u)$ and $(v,b)$. We consider only the situation with $u=v$, because otherwise the relation $(a,u)+(v,b) = (a,b) + (v,u)$ allows us to delete $(a,b)$. Moreover, we may assume that $u=x$, because otherwise the relations $(a,u)+(x,2) = (a,2)+(x,u)$ in $\cR$ and $(1,2)+(a,b)=(1,b)+(a,2)$ in $\cL$ would give a deletion of the pair $(a,2)$. Summing up, we consider the case where $\cR$ contains $(1,x,y)_{g_1}$, $(x,2)$, $(a,x)$ and $(x,b)$.

Let us determine in what flows $x$ appears with $g_1$ in $\cL$. If in a pair $(x,c)$ then we can use one of the relations $$(x,c)+(1,2) = (x,2)+(1,c) \qquad \qquad (x,c)+(a,b) = (x,b)+(a,c)$$ and delete $(x,2)$ or $(x,b)$. Hence we may assume that $x$ appears with $g_1$ in $\cL$ only in triples. If $(x,c,d)_{g_1}$ is such a triple then we try to use the relations
$$(x,c,d)_{g_1}+(1,2) = (x,2)+(1,c,d)_{g_1} \qquad (x,c,d)_{g_1} + (a,b) = (x,b) + (a,c,d)_{g_1}$$ and delete the same pairs as before. This fails only if $\{c,d\} = \{1,a\}$, so we may assume that $x$ appears with $g_1$ in $\cL$ only in triples $(x,1,a)_{g_1}$. In particular, there is at least one such triple. But then we can use the relation $(1,x,y)_{g_1} + (a,x) = (1,x,a)_{g_1}+(y,x)$ in $\cR$, which works because $x\neq y$ as they appear in the same triple, and delete $(x,1,a)_{g_1}$.

Hence in this case we can always find a deletion of a flow, after using suitable relations.

\subsection{Neither 1 nor $a$ appears in a $g_1$-triple and neither 2 nor $b$ is in a $g_2$-triple in $\cR$} Then there are pairs $(1,x)$, $(x',2)$, $(a,y)$ and $(y',b)$ in $\cR$. We may assume that $x=x'$ and $y=y'$, because otherwise we could use a relation of degree 2 to produce a pair $(1,2)$ or $(a,b)$ and cancel it with equal pair in $\cL$. Then, if $x\neq y$, we could use the relation $(1,x)+(y,b) = (1,b)+(x,y)$ in $\cR$ and the relation $(1,2)+(a,b) = (1,b)+(2,a)$ in $\cL$ and delete $(1,b)$. Hence we assume that~$\cR$ contains pairs $(1,x)$, $(x,2)$, $(a,x)$ and $(x,b)$.

Let us check in what flows~$x$ appears in~$\cL$. If it appears in a pair $(x,c)$ then we could use one of the relations $$(x,c) + (1,2) = (x,2) + (1,c) \qquad \qquad (x,c)+(a,b) = (x,b)+(a,c)$$ and delete either $(x,2)$ or $(x,b)$. Thus we may assume that $x$ does not appear with $g_1$ in any pair in $\cL$ and also, arguing in the same way, that~$x$ does not appear with~$g_2$ in any pair in $\cL$. This means that $\cL$ must contain a $g_1$-triple and a $g_2$-triple with~$x$, but by Lemma~\ref{different_triple_types} such a situation cannot happen. Hence this case must end with a deletion of a pair.

\section{At least two pairs on one side, three different indices}\label{section_three_indices}

Now we assume there are no two pairs on one side of the relation consisting of four pairwise different indices, but there are two pairs such that the set of their indices has three elements. Let 1, 2 and $a$ be pairwise different indices. There are two cases (up to swapping~$g_1$ with~$g_2$ and~$\cL$ with~$\cR$, and renaming indices):
\begin{enumerate}
\item[(5.1)] $\cL$ contains pairs $(1,2)$, $(1,a)$,
\item[(5.2)] $\cL$ contains pairs $(1,2)$, $(a,1)$.
\end{enumerate}

\subsection{Case $(1,2)$, $(1,a)$}\label{subsection_three_indices_1}

There are three different possibilities to consider:
\begin{enumerate}
\item[(5.1.1)] 2 or $a$ appears in $g_2$-triples in $\cR$ (thus, by Lemma~\ref{different_triple_types}, 1 is not contained in a $g_1$-triple in $\cR$);
\item[(5.1.2)] 1 appears in a $g_1$-triple in $\cR$ (thus, by Lemma~\ref{different_triple_types}, neither $2$ nor $a$ is contained in a $g_2$-triple in $\cR$);
\item[(5.1.3)] 1 does not appear in any $g_1$-triple in $\cR$ and neither 2 nor $a$ appears in a $g_2$-triple in $\cR$.
\end{enumerate}

\subsubsection{Assume that $\cR$ contains a triple $(2,x,y)_{g_2}$ and a pair $(1,t)$} Then we try to use the relation $(2,x,y)_{g_2}+(1,t) = (t,x,y)_{g_2}+(1,2)$ and delete $(1,2)$. Consider the situation when it is impossible: let $x=t$ and $(1,x) \in \cR$. We may assume $x \neq a$, otherwise $\cL$ and $\cR$ are not disjoint. Now consider flows with $x$ in $\cL$. If $x$ appears with $g_2$ in $\cL$ in a pair $(c,x)$ then we may swap elements in some two pairs as in Remark~\ref{exchanged_pairs} and delete $(1,x)$, because either $c \neq 2$ or $c \neq a$. Hence we may assume that $x$ appears with $g_2$ in $\cL$ just in triples.

Let $(x,c,d)_{g_2} \in \cL$. Then either we use one of the relations
$$(x,c,d)_{g_2} + (1,2) = (2,c,d)_{g_2}+(1,x) \qquad (x,c,d)_{g_2} + (1,a) = (a,c,d)_{g_2}+(1,x)$$
and delete $(1,x)$ or every $g_2$-triple in $\cL$ containing $x$ is of the form $(x,a,2)_{g_2}$. Now we check in what configurations 2 can occur in $\cR$ with $g_2$. If it occurs only in triples, then for $(2,e,f)_{g_2} \in \cR$ we try to use the relation $(2,e,f)_{g_2}+(1,x) = (x,e,f)_{g_2}+(1,2)$ and delete $(1,2)$. It is impossible only in the case where all $g_2$-triples in $\cR$ containing 2 contain also $x$, but then we finish the argument by applying Lemma~\ref{counting_argument}: since $x$ appears in $\cL$ with $g_2$ only in triples containing 2, and 2 appears in $\cR$ in $g_2$ only in triples containing~$x$, then we get a contradiction with the fact that $(1,2) \in \cL$.

This leaves us in the situation where there is a pair $(e,2)$ in~$\cR$. Then we may assume that $e=x$, because otherwise we swap elements in pairs with $(1,x)$ in~$\cR$ and delete $(1,2)$. But then $x$ must occur with $g_1$ in~$\cL$ and it has to be in a pair since we already have $g_2$-triples there. If $(x,f)\in \cL$ and $f \neq 1$ then we swap indices with $(1,2)$ and delete $(x,2)$. And if $(x,1) \in \cL$ then we use the relation $(x,a,2)_{g_2}+(x,1) = (x,a,1)_{g_2} + (x,2)$ in~$\cL$, hence $(x,2)$ can also be deleted.

Thus in this case we always get a pair to delete after using some relations.

\subsubsection{In $\cR$ there are $(1,x,y)_{g_1}$, $(z,2)$ and $(t,a)$} We may assume that $z=x$ (or $z=y$, but then we just swap $x$ with $y$), because if not then we could use the relation $(1,x,y)_{g_1}+(z,2) = (z,x,y)_{g_1} + (1,2)$ and delete $(1,2)$. (In the same way we see that $t \in \{x,y\}$, we use it later.)

First assume that both $x$ and $t$ occur in $\cL$ with $g_1$ in a pair, i.e. $(x,c), (t,d) \in \cL$ for some~$c,d$. Then we can swap elements in pairs as in Remark~\ref{exchanged_pairs} (either $(x,c)$ and $(1,2)$ or $(t,d)$ and $(1,a)$) and delete a pair, unless $c=d=1$. In this case we must have $(e,1) \in \cR$, because 1 has to appear with $g_2$ in~$\cR$, but there cannot be any $g_2$-triple by Lemma~\ref{different_triple_types}. This allows us to swap elements in pairs in $\cR$ and delete $(x,1)$ or $(t,1)$, since either $e\neq 2$ or $e \neq a$.

Hence the only possibility is that $x$ or $t$ occurs in $\cL$ with $g_1$ only in triples. Let~$x$ have this property -- the situation is symmetric since $t\in \{x,y\}$.
If $(x,f,h)_{g_1} \in \cL$ then we try to use the relation $(x,f,h)_{g_1}+(1,2) = (1,f,h)_{g_1}+(x,2)$ and delete $(x,2)$. This fails only if every $g_1$-triple in $\cL$ containing $x$ contains also 1. Consider possible flows with $g_1$ at index 1 in $\cR$.

If all of them are triples then we may assume that they contain $x$, because otherwise the relation $(1,i,j)_{g_1}+(x,2) = (x,i,j)_{g_1}+(1,2)$ can be applied and we have a deletion of $(1,2)$. Such a situation is impossible by Lemma~\ref{counting_argument}: if $x$ appears with $g_1$ in $\cL$ always in triples containing 1, and 1 appears with $g_1$ in $\cR$ only in triples containing $x$, then we get a contradiction with the fact that $(1,2) \in \cL$.

We are left with the case where 1 appears with $g_1$ in $\cR$ in a pair $(1,k)$. We try to swap elements in pairs in $\cR$ with $(x,2)$ or $(t,a)$ and delete $(1,2)$ or $(1,a)$. This fails only if $k=x=t$. But then we have $(1,x) \in \cR$, so~$x$ must occur in~$\cL$ with~$g_2$. It can be only in a pair because there is already a $g_1$-triple. Hence take $(m,x) \in \cL$. Now we can swap elements in pairs with $(1,2)$ or $(1,a)$ in $\cL$ in a way which allows to delete $(1,x)$.
Thus this case also always ends by a deletion of a pair.

\subsubsection{Here 1 does not appear in any $g_1$-triple in $\cR$ and neither 2 nor $a$ appears in a $g_2$-triple in $\cR$} Hence $\cR$ contains pairs $(1,x)$, $(y,2)$ and $(z,a)$. If $x\neq y$ or $x\neq z$ we can use one of the relations
$$(1,x)+(y,2)=(1,2)+(y,x)\qquad \qquad (1,x)+(z,a)=(1,a)+(z,x)$$ and delete $(1,2)$ or $(1,a)$. So we may assume that $\cR$ contains $(1,x)$, $(x,2)$ and $(x,a)$. We see that then~$x$ must appear with~$g_1$ in~$\cL$. If it appears in a pair $(x,c)$ such that $c\neq 1$ then we can use the relation $(x,c)+(1,2) = (x,2)+(1,c)$ and delete $(x,2)$. Hence if $x$ appears with $g_1$ in $\cL$ in a pair then it is $(x,1)$.

If such a pair belongs to $\cL$ then 1 must appear with $g_2$ in $\cR$. If there is $(1,s,t)_{g_2} \in \cR$ then we try to use the relation $(1,s,t)_{g_2} + (x,2) = (2,s,t)_{g_2}+(x,1)$ and delete $(x,1)$. It could fail only if $2 \in \{s,t\}$, but this situation contradicts the assumption that 2 does not appear in $g_2$-triples in $\cR$. And if there is $(s,1)\in \cR$,  we can use one of the relations $$(s,1)+(x,2) = (x,1)+(s,2)\qquad \qquad (s,1)+(x,a)=(x,1)+(s,a)$$ and delete $(x,1)$. Therefore we are left with the case where $x$ appears with $g_1$ in $\cL$ only in triples; there must be at least one such triple.

Then, because $(1,x) \in \cR$, we know that $x$ appears also with $g_2$ in $\cL$. It cannot be in a $g_2$-triple since currently we work under assumption that there are $g_1$-triples in~$\cL$, so it appears in a pair $(c,x)$. Since either $c \neq 2$ or $c \neq a$, we can use one of the relations $$(c,x)+(1,2) = (c,2)+(1,x)\qquad \qquad (c,x)+(1,a) = (c,a)+(1,x)$$ and delete $(1,x)$. 
Thus again in this case always a pair can be deleted after using certain relations.

\subsection{Case $(1,2)$, $(a,1)$}\label{subsection_three_indices_2} There are three subcases up to swapping sides of the relation or elements of~$\Z_3$, or renaming the indices. First two concern situations when at least one of $\{1,2,a\}$ occurs in a triple in~$\cR$, in the third one all these indices appear only in pairs in~$\cR$. Because of Lemma~\ref{different_triple_types}, up to swapping $g_1$ with $g_2$ there are just two different possibilities. In the last case we see what happens when there are no such triples.

\begin{rem}\label{rem_3pairs}
Assume that on one side of the relation three different indices $\alpha, \beta, \gamma$ appear in pairs and there is a pair $(\alpha,\beta)$.
Then either all the remaining pairs on this side are equal to $(\alpha,\beta)$, $(\gamma,\alpha)$ or $(\beta,\gamma)$, or we are in the situation which was already considered in section~\ref{section_four_indices} (four different indices in pairs) or~\ref{subsection_three_indices_1} (an index repeats on the same position in two different pairs).
\end{rem}

\subsubsection{First consider the case where~1 appears in a $g_1$-triple in~$\cR$} Then, by Lemma~\ref{different_triple_types}, 2 appears with $g_2$ in~$\cR$ only in pairs. So $\cR$ contains a triple $(1,x,y)_{g_1}$ and a pair $(z,2)$. The relation $(1,x,y)_{g_1} + (z,2) = (z,x,y)_{g_1} + (1,2)$, which leads to the deletion of $(1,2)$, cannot be used only if $z \in \{x,y\}$, so we may assume that~2 appears with $g_2$ in~$\cR$ only in pairs $(x,2)$ (by Remark~\ref{rem_3pairs} in this case $(y,2) \notin \cR$).

Also, 1 appears in $\cR$ with $g_2$ and it has to be only in pairs, since there already are $g_1$-triples. If $(s,1) \in \cR$ then either, by Remark~\ref{rem_3pairs}, we are in one of the previous cases or it must be equal to $(2,1)$, because we already have $(x,2)$ and $x\neq 1$. Hence we assume that $(2,1) \in \cR$. Then 2 appears in $\cL$ with $g_1$. If in a pair, then by Remark~\ref{rem_3pairs} it is $(2,a)$, so we can use the cubic relation
$$(1,2)+(a,1)+(2,a) = (1,a)+(2,1)+(a,2)$$
in $\cL$ and delete $(2,1)$.

Thus we only have to consider the case where 2 appears with~$g_1$ in~$\cL$ just in triples. Take such a triple $(2,c,d)_{g_1}$. If $a \notin \{c,d\}$ then we use the relation $(2,c,d)_{g_1}+(a,1) = (a,c,d)_{g_1}+(2,1)$ and delete $(2,1)$, so we may assume that if 2 appears in~$\cL$ with~$g_1$ then it is always in a triple containing~$a$.

If~$a$ appears in~$\cR$ with~$g_1$ in a pair $(a,t)$, then by Remark~\ref{rem_3pairs} this pair must be equal to $(x,2)$, because $a\notin \{1,2\}$. In particular, $x=a$ and $\cR$ contains $(1,a,y)_{g_1}$. Recall that $(2,a,d)_{g_1} \in \cL$. We use the following relations in $\cL$ and $\cR$ respectively:
$$(2,a,d)_{g_1} + (1,2) = (2,a,1)_{g_1}+(d,2)\qquad \qquad (1,a,y)_{g_1}+(2,1) = (1,a,2)_{g_1}+(y,1)$$
and delete a triple. This is always possible, because $2\neq d$ and $y \neq 1$, since these indices appear together in triples.

The last possibility is that $a$ appears in $\cR$ with $g_1$ just in triples and let $(a,s,t)_{g_1}$ be such a triple. Then either we can use the relation $(a,s,t)_{g_1}+(2,1) = (a,1)+(2,s,t)_{g_1}$ and delete $(a,1)$, or all such triples contain also 2. In the latter case we get a contradiction by Lemma~\ref{counting_argument}: 2 appears with $g_1$ in $\cL$ only in triples containing $a$, and $a$ appears with $g_1$ in $\cR$ only in triples containing 2, which contradicts the fact that $(a,1) \in \cL$.

Thus here we either reduce to one of the cases described in section~\ref{subsection_three_indices_1} (by Remark~\ref{rem_3pairs}) or find a deletion of a pair or a triple.

\subsubsection{Now consider the case where 1 does not appear in any $g_1$-triple and 2 appears in a $g_2$-triple in $\cR$}\label{subsection_three_2} Thus~$\cR$ contains a pair $(1,x)$ and a triple $(2,y,z)_{g_2}$. If $x\notin \{y,z\}$ then we can use the relation $(1,x)+ (2,y,z)_{g_2} = (1,2) + (x,y,z)_{g_2}$ and delete $(1,2)$, so we may assume that $x=z$ and $(2,x,y)_{g_2}\in \cR$. Also $a$ must appear with $g_1$ in $\cR$. It cannot be in a triple since there are $g_2$-triples already, so we have a pair $(a,s) \in \cR$. If $s=1$ then $(a,1)$ can be deleted, and if not, by Remark~\ref{rem_3pairs} we obtain $x=a$, i.e. we may assume that~$\cR$ contains $(1,a)$, $(a,s)$ and $(2,a,y)_{g_2}$.

Now we check in what flows 1 can appear with~$g_2$ in~$\cR$. If in a pair $(c,1)$ then by Remark~\ref{rem_3pairs} we have $c=s$ and we use the relation $$(1,a)+(a,s)+(s,1)=(1,s)+(a,1)+(s,a)$$ in $\cR$ and delete $(a,1)$. If 1 appears in a triple $(1,d,e)_{g_2} \in \cR$ and $s \notin \{d,e\}$ then we use the relation $(1,d,e)_{g_2}+(a,s) = (a,1)+(s,d,e)_{g_2}$ and delete $(a,1)$ again. Hence we may assume that 1 occurs with $g_2$ in $\cR$ only in triples containing $s$.

Since $(1,a) \in \cR$, there must be a flow containing~$a$ with~$g_2$ in~$\cL$. If~$a$ occurs in a pair $(f,a)$, then by Remark~\ref{rem_3pairs} $f=2$. Then we can again use the cubic relation between pairs, but this time in~$\cL$, and delete $(1,a)$. Thus we may assume that $a$ occurs with $g_2$ in $\cL$ only in triples. Moreover, if $(a,h,i)_{g_2}$ is such a triple, then either we can use the relation $(a,h,i)_{g_2}+(1,2) = (2,h,i)_{g_2}+(1,a)$ and delete $(1,a)$, or we may assume that $(a,2,i)_{g_2}\in \cL$ and every $g_2$-triple in $\cL$ containing~$a$ contains also~2.

In this situation we already have two configurations with 2 with $g_2$ in $\cL$, but only one in $\cR$. Thus 2 appears more times with $g_2$ in $\cR$. First assume it is only in triples and take $(2,j,k)_{g_2} \in \cR$. If $a \notin \{j,k\}$ then we can use the relation $(1,a) + (2,j,k)_{g_2} = (1,2) + (a,j,k)_{g_2}$ and delete $(1,2)$. Hence we may assume that a $g_2$-triple containing~2 always contains also~$a$, but this is impossible by Lemma~\ref{counting_argument} applied to occurrences of 2 and $a$ with $g_2$, because $(1,2) \in \cL$.

And if there is a pair $(m,2) \in \cR$, then by Remark~\ref{rem_3pairs} we have $m=a$ and $s=2$ (recall that $(1,a), (a,s) \in \cR$). In particular, in $\cR$ there are $(a,2)$ and $(1,2,e)_{g_2}$.
Then we use the following relations in $\cL$ and $\cR$ respectively
$$(a,2,i)_{g_2}+(a,1) = (a,2,1)_{g_2}+(a,i)\qquad \qquad (1,2,e)_{g_2}+(1,a)=(1,2,a)_{g_2}+(1,e)$$
and delete $(a,2,1)_{g_2}$. This is possible since $a\neq i$ and $1\neq e$, as they appear together in triples.

As before, this case ends by reducing to~\ref{subsection_three_indices_1} or by finding a pair or a triple to delete.

\subsubsection{Here we assume that 1 appears only in pairs in $\cR$, and $2$ with $g_2$ and $a$ with $g_1$ also appear only in pairs in $\cR$} Thus $\cR$ contains pairs $(1,x)$, $(x',2)$, $(y,1)$, and $(a,y')$. If $x\neq x'$ or $y\neq y'$ then we can swap elements in pairs and delete $(1,2)$ or $(a,1)$. Hence we may assume that $\cR$ contains $(1,x)$, $(x,2)$, $(y,1)$ and $(a,y)$. Now by Remark~\ref{rem_3pairs} either we are in the case~\ref{subsection_three_indices_1} or  we have $(y,1)=(2,1)$. In the latter case we can delete $(1,2)$ after applying the cubic relation $$(1,x)+(x,2)+(2,1) = (1,2)+(2,x)+(x,1).$$

\section{Only two different indices in pairs}\label{section_two_indices}

Assume that $(1,2) \in \cL$ and the only pairs in~$\cL$ are $(1,2)$ and possibly $(2,1)$. Moreover, in $\cR$ also at most two indices appear in pairs -- otherwise we are in one of the previous cases. Thus if $\cR$ contains pairs $(1,x)$ and $(y,2)$ then $x=2$ and $y=1$, so we have an immediate deletion. Hence we are left with two subcases, depending on which index from $(1,2)$ appears in~$\cR$ just in triples, since by Lemma~\ref{different_triple_types} both cannot be in triples. These subcases are symmetric, hence we consider just one of them.

Assume that $(1,x) \in \cR$ and $(2,s,t)_{g_2} \in \cR$. If $x \notin \{s,t\}$ then we can use the relation $(2,s,t)_{g_2}+(1,x) = (x,s,t)_{g_2}+(1,2)$ and delete $(1,2)$. Hence we may assume that if 2 appears with $g_2$ in $\cR$ then it is always in a triple containing~$x$, and there is at least one such triple, which implies that $x \neq 2$.

We check in what flows~$x$ appears with~$g_2$ in~$\cL$. It cannot be in a pair, because the only pairs which can be in~$\cL$ are $(1,2)$ and $(2,1)$, but by looking at flows in~$\cR$ we know that $x\notin \{1,2\}$. If there is a triple $(x,c,d)_{g_2} \in \cL$ such that $2\notin \{c,d\}$ then we can use the relation $(x,c,d)_{g_2} + (1,2) = (2,c,d)_{g_2} + (1,x)$ and delete $(1,x)$. Hence we may assume that $x$ appears with $g_2$ in $\cL$ only in triples containing 2. We apply Lemma~\ref{counting_argument} to get a contradiction: $x$ appears in $\cL$ with $g_2$ only in triples containing 2, and 2 appears in $\cR$ with $g_2$ only in triples containing $x$, but also we have $(1,2) \in \cL$.

Hence here we either show that we are actually in one of the previous cases or find a deletion of a pair. This case finishes the whole proof.

\bibliographystyle{plain}
\bibliography{Xbib}

\end{document}